\documentclass{svproc}
\usepackage{amsmath}
\usepackage{graphicx}

\usepackage{url}

\def\Cl{\operatorname{Cl}}
\def\Conv{\operatorname{Conv}}
\def\B{\operatorname{B}}
\def\Comp{{\operatorname{Comp}}}
\def\Int{\operatorname{Int}}
\def\deg{\operatorname{deg}}
\def\mpn{\operatorname{mpn}}
\def\smt{\operatorname{smt}}

\begin{document}
\mainmatter     
\title{Minimal Parametric Networks in Hyperspaces and their Properties}
\titlerunning{Minimal Parametric Networks in Hyperspaces}  
\author{Arsen Galstyan}
\institute{
School of Mathematics, Harbin Institute of Technology, Harbin, Heilongjiang 150001, China,\\
\email{ares.1995@mail.ru}.
}

\maketitle              

\begin{abstract}\footnote{The work of Arsen Galstyan was carried out at Harbin Institute of Technology with the support of the Postdoctoral Research Start-up Funds (Funding card number: AUGA5710026125).}

This work investigates minimal parametric networks in hyperspaces of closed subsets of metric spaces endowed with the Hausdorff distance. It is shown that the problems of finding such networks are nontrivial only within finiteness classes, where all Hausdorff distances between elements are finite. It is demonstrated that when studying the properties of minimal parametric networks, it is convenient to view their interior vertices as solutions of the Fermat--Steiner problem on the adjacent vertices. In this connection, already within the framework of the Fermat--Steiner problem, the structure of solution classes in hyperspaces of closed subsets of metric spaces is described. Results on the existence of $d$-far points in the case of convex boundary sets are also generalized. Namely, conditions are shown under which realizing one-sided Hausdorff distances holds.

\keywords{hyperspaces, convex sets, Hausdorff distance, Steiner problem, minimal networks}
\end{abstract}

\section*{Introduction}

The Steiner problem, which asks for a network of minimal total length connecting a given finite set of points in a metric space, is a classical optimization problem with numerous applications in science and engineering (see, e.g., \cite{Hwang}, \cite{Review}, \cite{Cheng}). The concept of a minimal parametric network (see Section~\ref{chap:min_net}) generalizes this problem by fixing the topology (the graph) of the network and optimizing the positions of its vertices in the ambient space (see, e.g., \cite{IT}, \cite{Branching}, \cite{Borodin}, \cite{Tropin}). 

From the very first formulation by Pierre de Fermat on the vertices of a triangle up to the works of the end of the last century, the Steiner problem was considered exclusively in finite-dimensional normed spaces (see, e.g., \cite{Hwang}, \cite{Du}, \cite{Bern}, \cite{Cieslik}). In \cite{IT}, Ivanov and Tuzhilin studied the Steiner problem on a connected complete Riemannian manifold. Moreover, in~\cite{IT} they considered the Steiner problem within the framework of searching for minimal parametric networks. In~\cite{Borodin} and~\cite{Bednov}, Borodin, Bednov and Strelkova studied the existence of minimal networks in Banach spaces. In recent years, Ivanov, Tuzhilin, Tropin and Galstyan wrote a series of papers investigating minimal parametric networks in hyperspaces (see Definition~\ref{dfn:hyperspace}) of compact subsets of finite-dimensional normed spaces (see \cite{Tropin}, \cite{ITT}, \cite{Tropin_2}, \cite{Gals_1}, \cite{Gals_2}). Moreover, some statements in these works were formulated in the context of hyperspaces of compact sets of metric spaces or proper metric spaces.

The theory of hyperspaces, which originated at the beginning of the 20th century thanks to the fundamental works of Hausdorff and Vietoris, has grown into a vast area of knowledge closely related to topology, functional analysis, and geometric measure theory (see, e.g., \cite{Nadler}, \cite{Blackburn}). The seminal 1942 paper of Kelley~\cite{Kelley} generated sustained interest in hyperspaces, and subsequent investigations showed (see, e.g., the book \cite{SetAn}) that these spaces provide a natural framework for studying set-valued analysis and variational problems.

The present paper continues and generalizes the research initiated in \cite{ITT} and \cite{Gals_1}. In Section~\ref{sec:tasks}, the problems of finding a minimal parametric network and a Steiner network are formulated in the context of extended metric spaces. An extended metric space differs from an ordinary metric space in that infinite distances are allowed. The problems are intentionally stated in such a space because the main goal of this work is to extend a number of results from \cite{ITT} and \cite{Gals_1} to hyperspaces of unbounded subsets, where the Hausdorff distance $d_H$ (see Definitions~\ref{distance_H} and~\ref{distance_H2}) can be equal to $\infty$. However, it is shown in Section~\ref{sec:tasks} that if the initial set of points to be connected by a minimal (parametric) network contains points that are at infinite distance from each other, then the problem of finding a minimal network becomes trivial. Namely, every network connecting such a set has length $\infty$ (see Proposition~\ref{prop:infty_len}). On the other hand, if all pairwise distances in the initial set are finite, then the infimum of network lengths is also finite (see Proposition~\ref{prop:finite_len_on_bound_points}), and in this case the problem of finding a minimal network becomes meaningful. Therefore, for the case of hyperspaces of unbounded subsets, the problems of finding minimal networks should be considered within finiteness classes (see Definition~\ref{dfn:finiteness}) of such hyperspaces. We denote the finiteness class of a hyperspace $\mathcal{P}_{*}(X)$ everywhere by $\mathcal{P}_{*}^f(X)$.

In Subsection~\ref{chap:connect}, a connection is also established between the particular Fermat--Steiner problem (see Definition~\ref{dfn:Fer_St}) and the problem of finding a minimal (not necessarily parametric) network. Namely, it is shown there that each interior vertex (see Definition~\ref{dfn:network_vert}) of a minimal network can be regarded as a solution of the Fermat--Steiner problem on its adjacent vertices (see Proposition~\ref{prop:Fer_St}). Thus, by discovering properties of solutions of the Fermat--Steiner problem in a specific type of metric spaces, we automatically extend these properties to all interior vertices of all minimal parametric networks, including minimal Steiner trees. This is why, even though properties of solutions of the Fermat--Steiner problem are studied in Section~\ref{Main}, that section is entitled ``Properties of Internal Vertices of Minimal Parametric Networks in Hyperspaces''. The main results of the paper are presented in Section~\ref{Main}.

Each solution $K$ of the Fermat--Steiner problem on a family of subsets \linebreak $\mathcal{M}=\{M_1, \ldots, M_n\}$ of a metric space $X$ corresponds to its own vector $d = \bigl(d_1=d_H(K, M_1), \ldots, d_n=d_H(K, M_n)\bigr)$. Different solutions $K_1$ and $K_2$ may correspond to the same vector $d$. Thus, for each vector $d$ we obtain a whole class of solutions $\Sigma_d(\mathcal{M})$, partially ordered by inclusion.

Subsection~\ref{chap:struct} describes the structure of such classes $\Sigma_d(\mathcal{M})$. Proposition~\ref{prop:matreshka} on the membership of intermediate subsets in the same solution class $\Sigma_d(\mathcal{M})$ generalizes Assertion~3.4 from~\cite{ITT} to the case of hyperspaces of closed subsets of a metric space. Proposition~\ref{prop:great} on the existence of a greatest element $K_d$ in $\Sigma_d(\mathcal{M})$ also extends the corresponding result Assertion~3.5 from~\cite{ITT} to the case of hyperspaces of closed subsets of a metric space. Proposition~\ref{prop:min_cl} on the existence of a minimal element in $\Sigma_d(\mathcal{M})$ generalizes the result Assertion~3.1 from~\cite{ITT} to the case of hyperspaces of closed subsets of a proper metric space.

Subsection~\ref{chap:one_dist} is devoted to a generalization of the result of Theorem~3 from~\cite{Gals_1} on the existence of $d$-far points (see Definition~\ref{dfn:far}) for solutions of the Fermat--Steiner problem. For finite-dimensional normed spaces and compact convex boundary sets, it was proved in~\cite{Gals_1} that there exists an index $i$ and a point $x \in M_i$ such that $U_{d_i}(x) \cap K_d = \emptyset$, i.e., $M_i$ has a $d$-far point. This implies $d_i = \sup_{x \in M_i} |x\, K_d|$. In other words, at least one of the summands in $S_{\mathcal{M}}(K_d)$ is one-sided, determined by the distances from the boundary set to the solution.

The main goal of Subsection~\ref{chap:one_dist} is to extend this result to a much broader setting. We drop the assumptions of finite-dimensionality and compactness and work with arbitrary closed convex subsets of a (possibly infinite-dimensional) normed space. The principal result (Corollary~\ref{cor:one_side}) states that if all $M_i$ lie in a finiteness class of closed convex sets, $\Sigma_d(\mathcal{M})\neq \emptyset$, and a certain technical condition on $K_d$ is satisfied, then, for any $K \in \Sigma_d(\mathcal{M})$, there exists an index $i$ for which $d_i = \sup_{x \in M_i} |x\, K|$.

It turns out that in hyperspaces of closed convex subsets of a two-dimensional normed space, the technical condition on $K_d$ from Corollary~\ref{cor:one_side} is satisfied automatically. Thus, we obtain Corollary~\ref{cor:one_side_2_dim}. Also, the technical condition on $K_d$ from Corollary~\ref{cor:one_side} is satisfied automatically in hyperspaces of closed bounded convex subsets of a normed space. Therefore we also have Corollary~\ref{cor:one_side_bound}.

We also address the additional question of whether all the distances $d_i$ can be ``realized'' from only one side: $d_i = \sup_{x \in M_i} |x\, K_d|$. Theorem~\ref{thm:one_side_two} shows that, in any connected metric space, for the greatest element $K_d\in \Sigma_d(\mathcal{M})$, there exists an index $i$ such that $d_i = \sup_{x \in K_d} |x\, M_i|$.

\subsection*{Acknowledgments}

The author wishes to express his sincere gratitude to Professor A. A. Tuzhilin and Professor A. O. Ivanov for their valuable remarks.

\section{Necessary Definitions and Statements}

For convenience, in the case of a metric space $(X,\rho)$, the distance between two points $a, b\in X$ will be denoted by $|a\,b|$ instead of $\rho(a, b)$. Also, instead of $(X,\rho)$, we will write $X$.

Let, $X$ be a metric space. The following notations will be used in many places in the text:
\begin{eqnarray*}
  \mathcal{P}_0(X)&:=&\{A\subset X\colon A \text{ is nonempty}\}; \\
  \mathcal{P}_{\Cl}(X)&:=&\{A\subset X\colon A \text{ is nonempty and closed}\}; \\
  \mathcal{P}_{\Cl, \B}(X)&:=&\{A\subset X\colon A\text{ is nonempty, closed and bounded}\}.
\end{eqnarray*}
If $X$ is a real normed space, then 
\begin{eqnarray*}
  \mathcal{P}_{\Conv}(X)&:=&\{A\subset X\colon A\text{ is nonempty and convex}\}; \\
  \mathcal{P}_{\Cl, \Conv}(X)&:=&\{A\subset X\colon A\text{ is nonempty, closed and convex}\}; \\
  \mathcal{P}_{\Cl, \B, \Conv}(X)&:=&\{A\subset X\colon A\text{ is nonempty, closed, bounded and convex}\}; \\
  \mathcal{P}_{\Comp, \Conv}(X)&:=&\{A\subset X\colon A\text{ is nonempty, compact and convex}\}.
\end{eqnarray*}

\begin{definition}\label{dfn:hyperspace}
The described sets of view $\mathcal{P}_{*}(X)$ endowed with Hausdorff distance \emph{(see Definitions~\ref{distance_H} and~\ref{distance_H2})} are called hyperspaces.
\end{definition}

\subsection{Intersection of Segment with Boundary}

\begin{lemma}[{\cite[Lemma~2]{Curves}}]
Let $M$ be a nonempty subset of a connected topological space $X$, not coinciding with $X$. Then $\partial M\ne\emptyset$.
\end{lemma}

We will also need the following

\begin{lemma}[{\cite[Lemma~3]{Wills}}]\label{light}
Let $B$ be a nonempty subset of a real normed space $X$. Then for any $a\in X\setminus\Int B$ and $b\in \Cl B$ we have $[a, b]\cap\partial B\ne\emptyset$.
\end{lemma}

\subsection{Definition and Properties of $r$-Neighborhoods}

\begin{definition}\label{definition:one}
For any subset $A$ of a metric space $X$ and any point $p\in X$, the distance from $p$ to $A$ is the value
\begin{equation}
|p\,A|=\inf\bigl\{|p\,a|:a\in A\bigr\}.
\end{equation}
In particular, for empty $A$ we have $|p\,A|=\infty$.
\end{definition}

\begin{remark}
Let $X$ be a metric space and $A\in \mathcal{P}_0(X)$. It is well known that the mapping $x\mapsto|x\,A|$ is continuous.
\end{remark}

\begin{proposition}\label{UnCont}
Let $X$ be a metric space, $A\in \mathcal{P}_0(X)$ and $x, y\in X.$ Then
\begin{equation}
|x\, A| \le |x\, y| + |y\, A|;
\end{equation}
\begin{equation}
\bigl| |x\, A| -  |y\, A| \bigr| \le |x\, y|.
\end{equation}
\end{proposition}

\begin{proof}

Let $z\in A.$ Then by the triangle inequality we have $|x\, z|\le |x\, y| + |y\, z|.$
Hence, for any $z\in A$ the inequality $\inf\limits_{p\in A} |x\, p| \le |x\, y| + |y\, z|$ holds. From this, $|x\, A| \le |x\, y| + \inf\limits_{z\in A}|y\, z|.$ Therefore,
\begin{equation}
|x\, A|\le |x\, y| + |y\, A|.
\end{equation}

Further, we have
\begin{equation}
|x\, A| - |y\, A|\le |x\, y|.
\end{equation}
At the same time, swapping $x$ and $y$, we obtain
\begin{equation}
|y\, A|\le |x\, y| + |x\, A|;
\end{equation}
\begin{equation}
|y\, A| - |x\, A|\le |x\, y|.
\end{equation}
Thus,
\begin{equation}
\bigl| |x\, A| -  |y\, A| \bigr| \le |x\, y|.
\end{equation}\qed

\end{proof}

\begin{definition}
Let $A$ be a subset of a metric space $X$ and $0\le r<\infty$. The subsets
\begin{equation}
B_r(A)=\{p\in X : |p\,A|\le r\},\ \ U_r(A)=\{p\in X : |p\,A|<r\}
\end{equation}
are called, respectively, the closed and the open ball with center at $A$ and radius $r$ or the closed and the open $r$-neighborhood of $A$.
\end{definition}

\begin{remark}
According to Definition~$\ref{definition:one}$, for any $0\le r < \infty$,
\begin{equation}
B_r(\emptyset)=U_r(\emptyset)=\emptyset.
\end{equation}
\end{remark}

In the case $A = \{a\}$, the notations $B_r\bigl(\{a\}\bigr)$ and $U_r\bigl(\{a\}\bigr)$ will be replaced for brevity by $B_r(a)$ and $U_r(a)$, respectively.

\begin{lemma}[{\cite[Lemma~4]{Curves}}]\label{lem:Closed}
Let $A$ be a subset of a metric space $X$. Then for any $0\le r < \infty$, the set $B_r(A)$ is closed and $B_0(A) = \Cl A$.
\end{lemma}

\begin{lemma}[{\cite[Corollary~6]{Curves}}]\label{lem:Opened}
For any subset $A$ of a metric space $X$ and any $0<r<\infty$, the set $U_r(A)$ is open.
\end{lemma}

\begin{lemma}[{\cite[Lemma~8]{Curves}}]\label{lem:convexity}
Let $X$ be a real normed space and $A\in \mathcal{P}_{\Conv}(X)$. Then for $0 < r<\infty$, the sets $B_r(A)$ and $U_r(A)$ are convex.
\end{lemma}

Also we will need the following

\begin{lemma}[{\cite[Lemma~11]{Curves}}]\label{sum_1}
Let $X$ be a real normed space $X$, $A\in \mathcal{P}_0(X)$ and $r, r' \in [0, \infty)$. Then $B_r\bigl(B_{r'}(A)\bigr) = B_{r+r'}(A)$.
\end{lemma}

\begin{lemma}\label{bound}
Let $X$ be a metric space, $A\in \mathcal{P}_0(X)$, $p\in X$ and $0\le r < \infty$. Then if $p\in \partial B_r(A)$, then $|p\, A| = r$.
\end{lemma}

\begin{proof}

Since $p\in\partial B_r(A)$, each neighborhood of $p$ contains points from $B_r(A)$ and $X\setminus B_r(A)$, hence $|p\,A|=r$ because the function $x\mapsto|x\,A|$ is continuous.\qed

\end{proof}

\begin{proposition}[{\cite[Proposition~12]{Curves}}]\label{as:ray}
Let $X$ be a real normed space, $A\in \mathcal{P}_{\Conv}(X)$, $x\not\in\Cl A$, and $p\in A$. Then the ray
\begin{equation}
\ell := (1-\lambda) p + \lambda x,\ \ \lambda\in [0, \infty),
\end{equation}
does not lie in any ball $B_r(A)$ for all $0\le r<\infty$.
\end{proposition}

\begin{proposition}\label{prop:Yens_more}
Let $X$ be a real linear space and $f\colon X\rightarrow \bbbr$ be a convex function. Let $a, b \in X$ and $t\ge 1.$ Then the following inequality holds for the point $x = (1-t)a + tb\colon$
\begin{equation}
f(x) \ge tf(b) - (t-1)f(a).
\end{equation}
\end{proposition}

\begin{proof}

We have 
\begin{equation}
x = a - ta + tb + b - b = b + (a-b) - t(a-b) = b + (1-t)(a-b).
\end{equation}
Also there exists $\lambda\in [0,1]$ such that $b = \lambda a + (1-\lambda) x.$ Express the quantity $\lambda$ through $t$:
\begin{equation}
\begin{aligned}
b&=\lambda a + (1-\lambda) \bigl(b + (1-t)(a-b)\bigr)=\\
  &=\lambda a + (1-\lambda) b + (1-\lambda)(1-t)(a-b)=\\
  &=\lambda a + b - \lambda b + (1 - t - \lambda + \lambda t) a - (1 - t - \lambda + \lambda t) b=\\
  &=(1-t+\lambda t)a - (-t+\lambda t)b.
\end{aligned}
\end{equation}
Thus, we obtain
\begin{equation}
(1-t+\lambda t)a - (1-t+\lambda t)b = 0;
\end{equation}
\begin{equation}
(1-t+\lambda t)(a - b) = 0.
\end{equation}
Since $a, b\in X$ are arbitrary, we get
\begin{equation}
1-t+\lambda t = 0;
\end{equation}
\begin{equation}
\lambda = 1 - \frac{1}{t}.
\end{equation}
Thus, $b = \lambda a + (1-\lambda) x = (1 - 1/t) a + (1/t) x.$ Since $f$ is a convex function and $t \ge 1$, then we have
\begin{equation}
f(b) \le (1 - 1/t) f(a) + (1/t) f(x).
\end{equation}
Hence, 
\begin{equation}
f(x) \ge tf(b) - (t-1)f(a).
\end{equation}\qed

\end{proof}

Let $X$ be a metric space and $A, B\in \mathcal{P}_0(X)$. Introduce a notation:
\begin{equation}
|A\,B| = \inf\limits_{a\in A} |a\,B|.
\end{equation}

\begin{proposition}\label{prop:free_space}
Let $A$ be a nonempty convex subset of a real normed space $X$ such that $A\neq X$, $0 < r < \infty$, $C\subset B_r(A)$ and $\bigl|C\, \partial B_r(A)\bigr| = \gamma > 0$. Then, for any $0 < \delta \le \min\{r, \gamma\}$, we have $C\subset B_{r-\delta}(A)$.
\end{proposition}

\begin{proof}

Assume the contrary, namely, let $c\in C$ and $c\notin B_{r-\delta}(A).$ Consider a sequence $\{a_n\}\subset A$ such that $|c\, a_n|\rightarrow |c\, A|.$ Draw rays $l_n$ starting from $a_n$ and passing through $c.$ Because of convexity of $A$, each ray $l_n$ does not lie in any ball $B_s(A)$ by Proposition~\ref{as:ray}. Hence, according to Lemma~\ref{light}, there exists a point $b_n\in l_n\cap \partial B_r(A)$ for each $n$. Thus, we obtain a sequence $\{b_n\}\subset \partial B_r(A)$, see Fig.~\ref{fig1}.
\begin{figure}[h]
\center{\includegraphics[scale=1.0]
{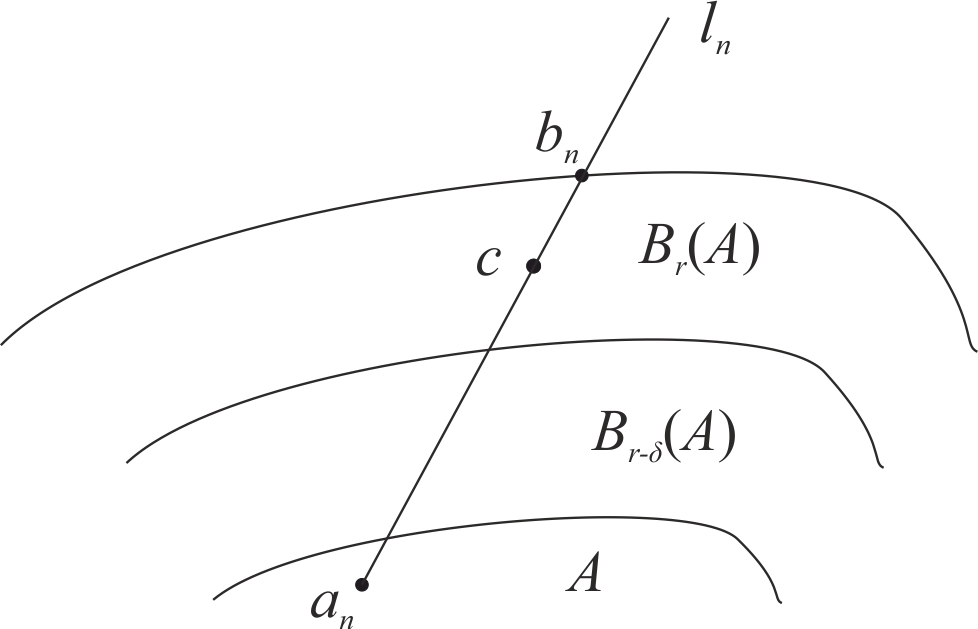}}
\caption{Sequence $\{b_n\}$.}
\label{fig1} 
\end{figure}
 
Note that for each $n$ there exists number $t_n\ge 1$ such that $b_n = (1-t_n)a_n + t_nc.$ Consider a distance from $c$ to $b_n:$
\begin{equation}\label{eq1:free_space}
\begin{aligned}
|c\, b_n|&= ||c - (1-t_n)a_n - t_nc|| = ||(1-t_n)c - (1-t_n)a_n|| = \\
&= |1 - t_n|\cdot ||c - a_n|| = (t_n - 1)\cdot |c\, a_n|.
\end{aligned}
\end{equation}

Give an upper bound for $t_n$. By Lemma~\ref{bound}, the equality 
\begin{equation}
r = |b_n\, A| = \Bigl|\bigl((1-t_n)a_n + t_nc\bigr)\,\, A\Bigr|
\end{equation}
holds, since $\{b_n\}\subset \partial B_r(A)$. The distance function to a convex set is well known to be convex. Thus, by Proposition~\ref{prop:Yens_more}, we obtain for $t_n\ge 1$ 
\begin{equation}
r = \bigl|b_n\, A\bigr| = \Bigl|\bigl((1-t_n)a_n + t_nc\bigr)\,\, A\Bigr| \ge t_n|c\, A| - (1-t_n)|a_n\, A| = t_n|c\, A|,
\end{equation}
since $\{a_n\}\subset A$. Hence, for all $n$, we have 
\begin{equation}
1\le t_n\le \frac{r}{|c\, A|}.
\end{equation}
Thus, we get from~(\ref{eq1:free_space})
\begin{equation}\label{eq2:free_space}
|c\, b_n| = (t_n - 1)\cdot |c\, a_n| \le \Bigl(\frac{r}{|c\, A|} - 1\Bigr)\cdot |c\, a_n| \rightarrow \Bigl(\frac{r}{|c\, A|} - 1\Bigr)\cdot |c\, A| = r - |c\, A|
\end{equation}
as $n\rightarrow \infty.$

On the other hand, by Proposition~\ref{UnCont}, we have $|b_n\, A| - |c\, A| \le |c\, b_n|.$ Hence, because of Lemma~\ref{bound}, the inequality $|c\, b_n|\ge r - |c\, A|$ holds for all $n$. Thus, by~(\ref{eq2:free_space}), we obtain 
\begin{equation}
|c\, b_n| \rightarrow r - |c\, A|
\end{equation}
as $n\rightarrow \infty.$ However, by the condition $\bigl|C\, \partial B_r(A)\bigr| = \gamma.$ It means that $\gamma\le |c\, b_n| \rightarrow r - |c\, A|$, that is 
\begin{equation}
|c\, A| + \gamma\le |c\, b_n| + |c\, A| \rightarrow r
\end{equation}
as $n\rightarrow \infty.$ So, we get
\begin{equation}
|c\, A| = |c\, b_n| + |c\, A| - |c\, b_n| \le |c\, b_n| + |c\, A| - \gamma \rightarrow r - \gamma \le r - \delta
\end{equation}
as $n\rightarrow \infty.$ Hence, $c\in B_{r-\delta}(A).$ This is a contradiction, so the proposition is proved.\qed

\end{proof}

\subsection{Hausdorff Distance}

\begin{definition}\label{distance_H}
Given a metric space $X$, we endow $\mathcal{P}_0(X)$ with the Hausdorff distance $d_H$ defined as follows$\colon$ for any $A,B\in\mathcal{P}_0(X)$,  we put
\begin{equation}
d_H(A, B) = \inf\bigl\{r:A\subset B_r(B),\,B\subset B_r(A)\bigr\}.
\end{equation}
\end{definition}

It is also well known that there is an equivalent definition of the Hausdorff distance (see, for example, the work~\cite[Proposition~5.1]{TuzLections}).

\begin{definition}\label{distance_H2}
\begin{equation}
d_H(A, B) = \max\bigl\{\sup\limits_{a\in A} |a\, B|,\, \sup\limits_{b\in B} |b\, A|\bigr\}.
\end{equation}
\end{definition}

Notice that $d_H$ is non-negative, symmetric, and satisfies the triangle inequality, but can be zero for $A\ne B$ and can be infinite.

\begin{lemma}\label{lem:equal}
Let $X$ be a metric spaces and $A, B\in \mathcal{P}_0(X).$ If $|A\, B| = d_H(A, B)$ then 
\begin{equation}
\sup\limits_{a\in A} |a\, B| = \sup\limits_{b\in B} |b\,  A|.
\end{equation}
\end{lemma}

\begin{proof}

We have $\inf\limits_{a\in A} |a\,B| = |A\,B| = d_H(A, B) = \max\bigl\{\sup\limits_{a\in A} |a\, B|,\, \sup\limits_{b\in B} |b\, A|\bigr\}.$ But
\begin{equation}
\inf\limits_{a\in A} |a\,B| = \inf\limits_{a\in A}\inf\limits_{b\in B} |a\,b| \le \sup\limits_{a\in A}\inf\limits_{b\in B} |a\,b| = \sup\limits_{a\in A} |a\, B|
\end{equation}
and
\begin{equation}
\inf\limits_{a\in A} |a\,B| = \inf\limits_{b\in B}\inf\limits_{a\in A} |a\,b| \le \sup\limits_{b\in B}\inf\limits_{a\in A} |a\,b| = \sup\limits_{b\in B} |b\, A|.
\end{equation}
Hence, $|A\, B| = \sup\limits_{a\in A} |a\, B| = \sup\limits_{b\in B} |b\,  A|.$\qed

\end{proof}

Introduce a notation. For any $A, B\in\mathcal{P}_0(X)$ and any $r\in\bigl(|A\,B|,\infty\bigr)\subset \bbbr$, we put $F_r(A, B)=B_r(A)\cap B$. Note that it holds $F_r(A, B)\in \mathcal{P}_0(X)$.  

In the next theorem, continuity of $F_r(A, B)$ in $r$ is understood with respect to $d_H$.

\begin{theorem}[{\cite[Theorem~20]{Curves}}]\label{cont_criterion}
Let $X$ be a real normed space, and $A, B\in\mathcal{P}_{\Conv}(X)$. Then the mapping $F_r(A,B)$ is continuous in $r$ on some interval $(Q,T)\subset\bigl(|A\,B|,\infty\bigr)$ if and only if for any $s_1,s_2\in(Q,T)$, it holds
\begin{equation}
d_H\bigl(F_{s_1}(A, B),F_{s_2}(A,B)\bigr)<\infty.
\end{equation}
\end{theorem}

\begin{theorem}[{\cite[Theorem~21]{Curves}}]\label{thm:2_dim}
Let $X$ be a real normed space, dimension of $X$ does not exceed $2$, and $A, B \in\mathcal{P}_{\Conv}(X)$. Then for any $r,\,s\in\bigl(|A\,B|,\infty\bigr)$, it holds
\begin{equation}
d_H\bigl(F_r(A,B),F_s(A,B)\bigr)<\infty.
\end{equation}
\end{theorem}

The following result is well known, see, e.g.,~\cite[Proposition~7.3.3]{Burago}.

\begin{proposition}\label{prop:metr}
Let $X$ be a metric space. Then
\begin{itemize}
\item $d_H\bigl(A, \Cl(A)\bigr) = 0$ for any $A\in \mathcal{P}_0(X);$
\item If $A$ and $B$ are closed subsets of the space $X$ and $d_H(A, B) = 0$, then $A = B.$
\end{itemize}
\end{proposition}

\begin{lemma}\label{lem:inter}
Let $X$ be a metric space and $A, B, C, D\in \mathcal{P}_0(X),$ with $A\subset B\subset C$ and $d_H(A, D) = d_H(C, D) = \alpha$. Then $d_H(B, D)\le \alpha$.
\end{lemma}

\begin{proof}

From $d_H(C, D) = \alpha$ and $B\subset C$ it follows that $\sup\limits_{b\in B} |b\, D|\le \sup\limits_{c\in C} |c\, D|\le \alpha.$ From $d_H(A, D) = \alpha$ and $A\subset B$ it follows that $\sup\limits_{d\in D} |d\, B|\le \sup\limits_{d\in D} |d\, A|\le \alpha.$ Hence $d_H(B, D)\le \alpha$.\qed

\end{proof}

\section{Minimal Parametric Network Problem}\label{sec:tasks}
\subsection{Minimal Networks}\label{chap:min_net}

In this subsection, we recall the necessary concepts from graph theory and fix the corresponding notation. More detailed information on graph theory can be found, for example, in~\cite{Graphs}.

\begin{definition}
A simple graph is a pair $(V, E)$ consisting of a finite set $V$ and a set $E$ of two-element subsets $\{u, v\} \subset V$. The elements of $V$ are called vertices of the graph, and the elements of $E$ are called edges of the graph. If $e = \{u, v\} \in E$, then the vertices $u$ and $v$ are said to be neighbors or adjacent and are joined by the edge $e$.
\end{definition}

Since we will only consider simple graphs henceforth, the word ``simple'' will be omitted.

Throughout, we will only consider undirected graphs, so we look at each edge $\{u, v\}\in E$ as an unordered pair of vertices, that is, we assume that $\{u, v\}$ and $\{v, u\}$ are the same edge.

We also emphasize that each edge is precisely a two-element set, that is, if $\{u, v\}$ is an edge of a graph, then $u\neq v.$

\begin{definition}
The number of vertices adjacent to a given vertex $u$ is called the degree of the vertex $u$ and is denoted by $\deg(u)$.
\end{definition}

\begin{definition}
A sequence of edges $\{v_1, v_2\}, \{v_2, v_3\}, \{v_3, v_4\},\ldots , \{v_n, v_{n+1}\}$ of a graph is called a route connecting the vertices $v_1$ and $v_{n+1}$, or a $v_1$-$v_{n+1}$ route.
\end{definition}

\begin{definition}
A graph is called connected if any two distinct vertices can be connected by a route.
\end{definition}

\begin{definition}
Let $G = (V, E)$ be a connected graph, and let $\mathcal{A} \subset V$. Then $G$ is said to join $\mathcal{A}$. The vertices in $\mathcal{A}$ are called boundary vertices, the set $\mathcal{A}$ itself is called the boundary of the graph, and the vertices in $V \setminus \mathcal{A}$ are called interior vertices.
\end{definition}

Thus, the boundary of a graph is some distinguished and fixed subset of its vertices. We will henceforth denote the boundary of the graph $G$ by $\partial G$.

\begin{definition}
Let $X$ be an arbitrary set. A function $d\colon X\times X\rightarrow [0, \infty]$ is called an extended metric if for any $x, y, z\in X$ the following conditions hold:
\begin{itemize}
\item[(1)] Non-negativity and non-degeneracy: $d(x, y) \ge 0,$ and $d(x, y) = 0$ if and only if $x = y.$
\item[(2)] Symmetry: $d(x, y) = d(y, x).$
\item[(3)] Triangle inequality: $d(x, z)\le d(x, y) + d(y, z).$
\end{itemize}
In this case, the set $X$ equipped with such a function $d$ is called an extended metric space.
\end{definition}

We emphasize that, in our understanding, any metric space is an extended metric space.

As in the case of a metric space, the distance between two points $a, b$ of an extended metric space $X$ will be denoted by $|a\,b|$ instead of $d(a, b)$.

\begin{definition} 
Let $G = (V, E)$ be a connected graph. A network of the type $G$ in an extended metric space $X$ is an arbitrary mapping $g\colon V\rightarrow X$. The graph $G$ is also called the parameterizing graph of $g$.
\end{definition}

Note that the mapping $g$ induces a corresponding mapping on the edges of the graph $G$ by the rule $g_e\colon \{v_i, v_j\} \mapsto \bigl\{g(v_i), g(v_j)\bigr\}$.

\begin{definition}\label{dfn:network_vert}
The elements of $g(V)$ are called vertices of the network, and the elements of $g_e(E)$ are called edges of the network. The elements of $g(\partial G)$ are called boundary vertices of the network, and the elements of $g(V\setminus \partial G)$ are called interior vertices of the network. If $u, v\in V$ were adjacent vertices of the graph $G$ then we call $g(u), g(v)$ adjacent vertices of the network $g$.
\end{definition}

\begin{definition}
An edge $\bigl\{g(v_i), g(v_j)\bigr\}$ of the network $g$ is called degenerate if $g(v_i) = g(v_j);$ otherwise, the edge is called non-degenerate. A network all of whose edges are non-degenerate will also be called non-degenerate.
\end{definition}

The length of an edge $\bigl\{g(v_i), g(v_j)\bigr\}$ of the network $g$ is defined as the distance $\bigl|g(v_i)\, g(v_j)\bigr|$, and the length of the network $g$ is defined as the sum of the lengths of all its edges:
\begin{equation}
|g| := \sum\limits_{\{v_i, v_j\}\in E} \bigl|g(v_i)\, g(v_j)\bigr|.
\end{equation}

Note that since $X$ is an extended metric space, the length of an edge, and hence the length of a network in $X$, may be equal to $\infty$.

Let a boundary $\partial G$ be a subset of an extended metric space $X$. We shall assume that in this case all networks $g$ satisfy the following condition: \textbf{the restriction of {\boldmath $g$} to {\boldmath $\partial G$} is the identity mapping}. Thus, each such network $g$ is uniquely determined by the images of the interior vertices of the parametrizing graph $G$.

\begin{definition}
Graphs $G = (V, E)$ and $H = (W, F)$ are called isomorphic if there exists a bijective map $\phi\colon V\rightarrow W$ such that $\{u, v\}\in E$ if and only if $\bigl\{\phi(u), \phi(v)\bigr\}\in F$. Such a mapping $\phi$ is called an isomorphism of the graphs $G$ and $H$.
\end{definition}

\begin{definition}
Let a connected graph $G_1 = (V_1, E_1)$ parameterizing a network $g_1$ and a connected graph $G_2 = (V_2, E_2)$ parameterizing a network $g_2$ be given and $\mathcal{A} = \partial G_1 = \partial G_2$. These networks are said to be of the same type $($parameter$)$, say $G_1$, if there exists an isomorphism $\phi\colon V_1\rightarrow V_2$ between the graphs $G_1$ and $G_2$ which is the identity on $\mathcal{A}$.
\end{definition}

For a fixed boundary $\mathcal{A}$ in an extended metric space $X$, consider a connected graph $G = (V, E)$ joining $\mathcal{A}$, that is $\partial G = \mathcal{A}\subset X$. Denote by $[G, \mathcal{A}]$ the set of all networks of this type $G$ in the extended metric space $X$. By the convention made above, all networks $g\in [G, \mathcal{A}]$ satisfy the condition $g\big|_{\partial G = \mathcal{A}} = \text{id}$. 

\begin{definition}\label{dfn:mpn}
A network from $[G, \mathcal{A}]$ that has the smallest possible length among all networks in $[G, \mathcal{A}]$ is called a minimal parametric network of type $G$ joining $\mathcal{A}$.
\end{definition}

A minimal parametric network does not always exist. Nevertheless, the infimum of the quantities $|g|$ over all networks $g\in [G, \mathcal{A}]$ always exists. 

\begin{definition}
The quantity 
\begin{equation}
\mpn[G, \mathcal{A}] = \inf\limits_{g\in [G, \mathcal{A}]} |g|
\end{equation}
is called the length of the minimal parametric network of type $G$ joining $\mathcal{A}$.
\end{definition}

\begin{definition}
The length of the minimal Steiner tree joining a subset $\mathcal{A}$ is the value
\begin{equation}
\smt(\mathcal{A}) = \inf\limits_{G\colon \partial G = \mathcal{A}} \mpn[G, \mathcal{A}].
\end{equation}
\end{definition}

\begin{proposition}\label{prop:nondeg}
Let $X$ be an extended metric space and $\mathcal{A}\subset X$. If there exist a graph $G$ joining $\mathcal{A}$ and a network $g\in [G, \mathcal{A}]$ such that $|g| = \smt(\mathcal{A})$, then there also exist a graph $G'$ joining $\mathcal{A}$ and a non-degenerate network $g'\in [G', \mathcal{A}]$ such that $|g'| = \smt(\mathcal{A})$.
\end{proposition}

\begin{proof}

The boundary $\mathcal{A}$ of a graph $G=(V, E)$ is a subset of the space $X$, so all its elements are distinct by definition. By convention, $g|_{\mathcal{A}} = \text{id}$. Therefore, if a network $g$ is degenerate, that is, if there exist adjacent vertices $u, v\in V$ such that $g(u)=g(v)$, then at least one of the vertices $u$ or $v$ is an interior vertex of the graph $G$.

Without loss of generality, let $u$ be interior. We construct a graph $G'$ by means of the following algorithm.
\begin{itemize}
\item[(1)] Replace all edges of the form $\{u, x\}$ in $E$ with $x\neq v$ by edges $\{v, x\}$;
\item[(2)] Remove from $E$ the edge joining the vertices $u$ and $v$, and then remove the vertex $u$ from $V$;
\item[(3)] If there still remain $u', v'\in V$ such that $g(u')=g(v')$, then determine which of these two vertices is interior, and return to step (1) of the algorithm.
\end{itemize}

The sets $V$ and $E$ transformed by the above algorithm are denoted by $V'$ and $E'$, respectively. Thus, the result of the algorithm is a connected graph $G' = (V', E')$ joining the set $\mathcal{A}\subset X$. Note that the corresponding network $g'\in [G', \mathcal{A}]$ is, by construction, a non-degenerate network, and $|g'| = |g| = \smt(\mathcal{A}).$\qed

\end{proof}

The number of elements in $M$ is denoted by $\# M$.

\begin{definition}\label{dfn:mst}
Let $X$ be an extended metric space, let a graph $G = (V, E)$ join $\mathcal{A}\subset X$, let a network $g\in [G, \mathcal{A}]$ be non-degenerate and $|g| = \smt(\mathcal{A})$. If there does not exist a graph $G' = (V', E')$ joining $\mathcal{A}$ and a non-degenerate network $g'\in [G', \mathcal{A}]$ such that $|g'| = \smt(\mathcal{A})$ and $\# (V'\setminus \mathcal{A}) < \# (V\setminus \mathcal{A})$, then $g$ is called a minimal Steiner tree joining $\mathcal{A}$.
\end{definition}

\begin{remark}\label{rk:tree}
Note that the parametrizing graph of a minimal Steiner tree $g$ joining $\mathcal{A}$ is a tree in case of $|g| < \infty$, that is, a connected graph without cycles. Indeed, if the parametrizing graph of a minimal Steiner tree has a cycle then one of the edges of this cycle could be removed, which would reduce the length of the original network. But then $g$ would not be a minimal Steiner tree by the definition.
\end{remark}

\begin{proposition}\label{prop:infty_len}
Let $X$ be an extended metric space, $\mathcal{A}$ be a finite subset of $X$ and $G$ be a connected graph joining $\mathcal{A}$. If there exist points $a, b\in \mathcal{A}$ such that $|a\, b| = \infty$, then $\mpn [G, \mathcal{A}] = \infty$.
\end{proposition}

\begin{proof}

Let $a, b \in \mathcal{A}$ be such that $|a\, b| = \infty$. Since the graph $G$ is connected, there exists a path $\{v_0 = a, v_1\}, \{v_1, v_2\},\ldots , \{v_{m-1}, v_m = b\}$ connecting the vertices $a$ and $b$. Consider an arbitrary network $g\in [G, \mathcal{A}]$. Recall that $g|_{\mathcal{A}} = \text{id}$. Hence, by the triangle inequality applied iteratively, we obtain
\begin{equation}
\infty = |a\, b| = \bigl|g(a)\, g(b)\bigr| \le \sum_{i=1}^m \bigl|g(v_{i-1})\, g(v_i)\bigr|.
\end{equation}
Thus, any network $g\in [G, \mathcal{A}]$ contains an edge of infinite length, hence $|g| = \infty$. Consequently, $\mpn [G, \mathcal{A}] = \inf\limits_{g\in [G, \mathcal{A}]} |g| = \infty.$\qed

\end{proof}

\begin{remark}\label{rk:arbitr}
Thus, by Proposition~\ref{prop:infty_len}, if there exist $a, b\in \partial G = \mathcal{A}\subset X$ such that $|a\, b| = \infty$ then any network from $[G, \mathcal{A}]$ is minimal parametric network of type $G$ joining $\mathcal{A}$. 
\end{remark}

Since a graph $G$ is arbitrary in Proposition~\ref{prop:infty_len}, we obtain

\begin{corollary}\label{cor:smt_infty}
Let $X$ be an extended metric space and $\mathcal{A}$ be a finite subset of $X$. If there exist points $a, b\in \mathcal{A}$ such that $|a\, b| = \infty$, then $\smt(\mathcal{A}) = \infty$.
\end{corollary}

\begin{remark}\label{rk:solut}
Note that, since $g\big|_{\partial G = \mathcal{A}} = \text{id}$, then $\bigl[G = (\mathcal{A}, E), \mathcal{A}\bigr]$ consists of an unique network $g$, that is $\bigl[(\mathcal{A}, E), \mathcal{A}\bigr] = \{g\}$. Thus, by Corollary~\ref{cor:smt_infty}, if there exist $a, b\in \mathcal{A}\subset X$ such that $|a\, b| = \infty$ and a graph $G = (\mathcal{A}, E)$ is connected then, due to Definition~\ref{dfn:mst}, the unique network $g$ from $[G, \mathcal{A}]$ is minimal Steiner tree joining $\mathcal{A}$. Consequently, henceforth we assume that in this case the graph $G = (\mathcal{A}, E)$ is also a tree.
\end{remark}

\begin{proposition}\label{prop:finite_len_on_bound_points}
Let $X$ be an extended metric space, $\mathcal{A}$ be a finite subset of $X$ and $G$ be a connected graph joining $\mathcal{A}$. If for any points $a, b\in \mathcal{A}$ we have $|a\, b| < \infty$, then $\mpn [G, \mathcal{A}] < \infty$.
\end{proposition}

\begin{proof}

Let us fix a vertex $a\in \mathcal{A}$ and consider such network $g'\in [G, \mathcal{A}]$:
\begin{equation}
g'(v) = \left\{
\begin{aligned}
    & v, && \text{if } v\in \mathcal{A} \\
    & a,  && \text{if } v\in V\setminus \mathcal{A}.
\end{aligned}
\right.
\end{equation}
Since, by condition, for any points $a, b\in \mathcal{A}$ we have $|a\, b| < \infty$, then $|g'| < \infty$. So $\mpn [G, \mathcal{A}] = \inf\limits_{g\in [G, \mathcal{A}]} |g| \le |g'| < \infty.$\qed

\end{proof}

Note that on an extended metric space $X$ one can naturally introduce a binary relation $\sim$, which is an equivalence relation: $x \sim y$ if and only if $|x\, y| < \infty$. 

\begin{definition}\label{dfn:finiteness}
The equivalence class of the relation $\sim$ on $X$ is called a finiteness class.
\end{definition}

\begin{remark}
According to Remark~\ref{rk:arbitr}, Remark~\ref{rk:solut}, and Proposition~\ref{prop:finite_len_on_bound_points}, the problem of finding a minimal parametric network or a minimum Steiner tree can be nontrivial only within a finiteness class. Moreover, each finiteness class is a metric space. \textbf{Therefore, henceforth all statements will be formulated only for metric spaces.}
\end{remark}

The following two propositions, as well as Remark~\ref{rk:finite_set}, can be found, for example, in~\cite[Lemma~1 and Lemma~2]{Bednov}. However, there they are proved for the case when $X$ is a real normed space. In fact, the proofs for the case when $X$ is a metric space are identical. We present these proofs for completeness.

\begin{proposition}\label{prop:deg_three}
Let $X$ be a metric space, $\mathcal{A}$ be a finite subset of $X$, $G$ be a connected graph joining $\mathcal{A}$ and $g\in [G, \mathcal{A}]$ be a minimal Steiner tree joining $\mathcal{A}$. Then any interior vertex of $G$ has degree at least $3$.
\end{proposition}

\begin{proof}

Suppose that in $G = (V, E)$ there exists an interior vertex $v\in V$ of degree less than $3$. The graph $G$ is connected, so the case $\deg(v) = 0$ is impossible.

The case $\deg(v) = 1$ is also impossible. Indeed, if $\deg v = 1$, then there exists exactly one vertex $u\in V$ to which $v$ is connected. Remove the edge $\{u, v\}$ from $E$ and remove the vertex $v$ from $V$. Denote the resulting sets of edges and vertices by $E'$ and $V'$, respectively. But then, due to the non-degeneracy of the network $g$, the new network $g' = g|_{V'}\in \bigl[G' = (V', E'), \mathcal{A}\bigr]$ will have a weight less than $|g|$, which contradicts the minimality of $g$.

Show that the case $\deg(v) = 2$ is also impossible. The vertex $v$ is connected to exactly two vertices of $V$, which we denote by $u$ and $w$. By the triangle inequality in the metric space $X$ we have:
\begin{equation}
\bigl|g(u)\, g(w)\bigr|\le |g(u)\, g(v)| + |g(v)\, g(w)|.
\end{equation}
Remove the edges $\{u, v\}$, $\{v, w\}$ from $E$, add the edge $\{u, w\}$ to $E$, and remove the vertex $v$ from $V$. Denote the resulting sets of edges and vertices by $E'$ and $V'$, respectively. But then the network $g' = g|_{V'}\in \bigl[G' = (V', E'), \mathcal{A}\bigr]$ will have a weight not exceeding $|g|$. Again we arrive at a contradiction with the definition of a minimal Steiner tree joining $\mathcal{A}$. Consequently, $\deg(v)\ge 3$.\qed

\end{proof}

\begin{proposition}\label{prop:int_vert}
Let $X$ be a metric space, a connected graph $G = (V, E)$ join $\mathcal{A}\subset X$, $\# \mathcal{A} = n < \infty$ and $g\in [G, \mathcal{A}]$ be a minimal Steiner tree joining $\mathcal{A}$. Then the graph $G$ has at most $n-2$ interior vertices.
\end{proposition}

\begin{proof}

Let $k$ be the number of internal vertices of the graph $G$. Then $\# V = n + k$ is the total number of vertices. Since, by Remark~\ref{rk:tree}, $G$ is a tree, the fundamental relation 
\begin{equation}
\# E = \# V - 1 = n + k - 1 
\end{equation}
holds. It is also well known that the sum of the degrees of all vertices of any undirected graph equals twice the number of edges: 
\begin{equation}
\sum\limits_{v\in V} \deg(v) = 2\cdot \#E = 2(n + k -1).
\end{equation}
On the other hand, since the graph $G$ is connected, for every vertex $v\in V$ we have $\deg(v) \ge 1.$ Moreover, if $v\in V\setminus \mathcal{A}$, then by Proposition~\ref{prop:deg_three} we have $\deg(v) \ge 3.$ Hence 
\begin{equation}
\sum\limits_{v\in V} \deg(v) \ge n + 3k.
\end{equation}
Consequently, $2(n + k -1)\ge n + 3k.$ Finally, we obtain $n - 2\ge k.$\qed

\end{proof}

\begin{remark}\label{rk:finite_set}
By Cayley's formula, for every positive integer $m$, the number of trees on $m$ labeled vertices is $m^{m-2}$. Consequently, in view of Proposition~\ref{prop:int_vert}, we can assume that in the expression $\inf\limits_{G\colon \partial G = \mathcal{A}} \mpn[G, \mathcal{A}]$ the infimum is taken over a finite set of graphs. Therefore, the infimum can be replaced by a minimum, that is, we can henceforth write 
\begin{equation}
\smt(\mathcal{A}) = \min\limits_{G\colon \partial G = \mathcal{A}} \mpn[G, \mathcal{A}].
\end{equation}
\end{remark}

\subsection{Existence Theorems}

\begin{definition}
A proper metric space is a metric space in which every closed bounded subset is compact.
\end{definition}

In the book~\cite{IT}, a proof is given of the existence of a minimal parametric network, as well as of a minimal Steiner tree, on a connected complete Riemannian manifold, see~\cite[Chapter~2, Paragraph~2, Section~2.1, Theorem~2.1 and Corollary~2.1]{IT}. As noted by Tropin in his work~\cite{Tropin}, the idea of the proof of Theorem~2.1 (and also Corollary~2.1) from~\cite{IT} remains exactly the same if one replaces the Riemannian manifold by a proper metric space (indeed, by the Hopf--Rinow theorem, connected Riemannian manifold is a proper metric space if and only if it is a complete metric space). To be complete, we will provide this proof.

\begin{theorem}\label{thm:exist_mpn}
Let $X$ be a proper metric space and $\mathcal{A}$ be a finite subset of $X$. Then for any connected graph $G = (V, E)$ with boundary $\mathcal{A}$ there exists a minimal parametric network of type $G$ joining $\mathcal{A}$.
\end{theorem}

\begin{proof}

Consider a sequence $\{g_n\}\subset [G, \mathcal{A}]$ such that $\lim\limits_{n\rightarrow \infty} |g_n| \rightarrow \mpn[G, \mathcal{A}]$. Without loss of generality, we may assume that the sequence $|g_n|$ is decreasing. Hence, $\bigl\{|g_n|\bigr\}$ is a bounded sequence.

For each vertex $v\in V$, the sequence $g_n$ generates a corresponding sequence of points $\bigl\{g_n(v)\bigr\}\subset X$. In the case $v\in \mathcal{A}$, this sequence is constant, because $g_n(v)=v$. Fix a point $a\in \mathcal{A}$. Since the graph $G$ is connected, for any $v\in V$ there exists a route $\{v_1 = v, v_2\}, \{v_2, v_3\},\ldots , \{v_{k-1}, v_k = a\}$ connecting $v$ and $a$. Applying the triangle inequality successively, we obtain
\begin{equation}
\bigl|g_n(v)\, a\bigr|\le \sum\limits_{i=1}^{k-1} \bigl|g(v_i)\, g(v_{i+1})\bigr| \le |g_n| \le |g_1|.
\end{equation}
Since the vertex $v$ is arbitrary, then there exists a number $0 < r < \infty$ such that 
\begin{equation}
\bigcup\limits_{v\in V} \bigcup\limits_{n = 1}^{\infty} \bigl\{g_n(v)\bigr\} \subset B_r(a).
\end{equation}
As $X$ is a proper metric space, $B_r(a)$ is compact. Therefore, for each $v\in V$, there exists a subsequence $\bigl\{g_{n_m}(v)\bigr\}\subset \bigl\{g_n(v)\bigr\}$ converging to some point $v_g\in B_r(a)$. 

Consider the network $\widetilde{g}\in [G, \mathcal{A}]$ defined as follows:
\begin{equation}
\widetilde{g}\colon v\mapsto v_g.
\end{equation}
If $a\in \mathcal{A}$, then we have $\widetilde{g}(a) = a$. We show that $|\widetilde{g}| = \mpn[G, \mathcal{A}]$.

The distance function is jointly continuous in both variables, so for any edge $\{u, v\}\in E$ we have
\begin{equation}
\lim\limits_{m\rightarrow \infty} \bigl|g_{n_m}(u)\, g_{n_m}(v)\bigr| = \bigl|\widetilde{g}(u)\, \widetilde{g}(v)\bigr|.
\end{equation}
Hence,
\begin{multline}
|\widetilde{g}| = \sum\limits_{\{u, v\}\in E} \bigl|\widetilde{g}(u)\, \widetilde{g}(v)\bigr| = \sum\limits_{\{u, v\}\in E} \lim\limits_{m\rightarrow \infty} \bigl|g_{n_m}(u)\, g_{n_m}(v)\bigr| = \\ = \lim\limits_{m\rightarrow \infty} \sum\limits_{\{u, v\}\in E} \bigl|g_{n_m}(u)\, g_{n_m}(v)\bigr| = \lim\limits_{m\rightarrow \infty} \bigl|g_{n_m}\bigr| = \mpn[G, \mathcal{A}].
\end{multline}
Thus, $\widetilde{g}$ is a minimal parametric network of type $G$ joining $\mathcal{A}$.\qed

\end{proof}

From Theorem~\ref{thm:exist_mpn} and Remark~\ref{rk:finite_set} we obtain

\begin{corollary}
Let $X$ be a proper metric space and let $\mathcal{A}$ be a finite subset of $X$. Then there exists a minimal Steiner tree joining $\mathcal{A}$.
\end{corollary}

\begin{theorem}\label{thm:cl_boud_h}
A space $X$ is a proper metric space if and only if $\mathcal{P}_{\Cl, \B}(X)$ is a proper metric space.
\end{theorem}

\begin{proof}

Let $\mathcal{P}_{\Cl, \B}(X)$ be a proper metric space. Note that $X\subset \mathcal{P}_{\Cl, \B}(X)$. Consider the isometric embedding
\begin{equation}
\pi\colon X\rightarrow \mathcal{P}_{\Cl, \B}(X);
\end{equation}
\begin{equation}
\pi\colon x\mapsto \{x\};
\end{equation}
\begin{equation}\label{eq:reduce_dH}
d_H\bigl(\pi(x),\pi(y)\bigr)=|x\, y|,
\end{equation}
for all $x,y\in X$. Let $A\subset X$ be closed and bounded. Then $\pi(A)=\bigl\{\{x\}\colon x\in A\bigr\}$ is obviously bounded. 

Show that $\pi(A)$ is closed. If a sequence $\{\{x_i\}\}\subset \pi(A)$ converges with respect to the metric $d_H$ then its limit should be a single-point subset $\{x'\}$. Because of~(\ref{eq:reduce_dH}), we obtain that $x'$ is the limit of the sequence $\{x_i\}\subset A$. Since $A$ is closed, then $x'\in A$ and so $\{x'\}\in \pi(A)$. Consequently, $\pi(A)$ is closed and bounded and so $\pi(A)$ is compact. Since an isometry is a homeomorphism onto its image, $A$ is also compact. Therefore $X$ is a proper metric space.

Conversely, suppose $X$ is a proper metric space. Let $\mathcal{A}\subset \mathcal{P}_{\Cl, \B}(X)$ be closed and bounded. We show that $\mathcal{A}$ is compact. Because $\mathcal{A}$ is bounded, its diameter with respect to $d_H$ equals some number $0\le r < \infty$. Hence there exists $A_0\in \mathcal{A}$ such that for every $A_1\in \mathcal{A}$ we have, with respect to the metric of $X$, that $A_1\subset B_r(A_0)\subset X$. Since $X$ is proper, $B_r(A_0)$ is compact. Then, by~\cite[Theorem~7.3.8]{Burago}, the space $\mathcal{P}_{\Cl, \B}\bigl(B_r(A_0)\bigr)$ is also compact.

For every $A_1\in \mathcal{A}$ we have $A_1\subset B_r(A_0)$, so $A_1\in\mathcal{P}_{\Cl, \B}\bigl(B_r(A_0)\bigr)$. The subset $\mathcal{A}$ is closed in $\mathcal{P}_{\Cl, \B}(X)$, hence it is closed in $\mathcal{P}_{\Cl, \B}\bigl(B_r(A_0)\bigr)$ as well. Compactness of $\mathcal{P}_{\Cl, \B}\bigl(B_r(A_0)\bigr)$ then implies compactness of $\mathcal{A}$.\qed

\end{proof}

\subsection{Connection with Fermat--Steiner Problem}\label{chap:connect}

\begin{proposition}\label{prop:Fer_St}
Let $X$ be a metric space, $\mathcal{A}$ be a finite subset of $X$ and $G = (V, E)$ be a connected graph joining $\mathcal{A}$. Let $v\in V\setminus \mathcal{A}$ and $\{u_1, \ldots, u_n\}\subset V$ be a set of all vertices adjacent to $v$ in $G$. If $g\in [G, \mathcal{A}]$ is a minimal parametric network, then $g(v)$ realizes the minimum of the sum of distances in $X$ to the elements $g(u_1), \ldots, g(u_n)$.
\end{proposition}

\begin{proof}

Consider an arbitrary point $y\in X$. Construct a new network $g'\in [G, \mathcal{A}]$ that coincides with $g$ on all vertices except $v$, and at vertex $v$ takes the value $g'(v)=y$. By minimality of $g$, we have $|g|\le |g'|$. Consequently,
\begin{eqnarray}
\begin{aligned}
0\le |g'|-|g| = \sum_{u\in \{u_1, \ldots, u_n\}}\bigl(|g'(v)\,g(u)|-|g(v)\,g(u)|\bigr) = \\ 
= \sum_{u\in \{u_1, \ldots, u_n\}}\bigl(|y\,g(u)|-|g(v)\,g(u)|\bigr).
\end{aligned}
\end{eqnarray}

Thus, for any $y\in X$, we get $\sum_{u\in \{u_1, \ldots, u_n\}}|y\,g(u)|\ge \sum_{u\in \{u_1, \ldots, u_n\}}|g(v)\,g(u)|$, i.e., $g(v)$ is a global minimum point of the function 
\begin{equation}
S_{\{u_1, \ldots, u_n\}}(y)=\sum_{u\in \{u_1, \ldots, u_n\}}|y\,g(u)|.
\end{equation}
Hence, $g(v)$ realizes the minimum of the sum of distances to the images of adjacent vertices.\qed

\end{proof}

\begin{definition}\label{dfn:Fer_St}
Let $X$ be a metric space and $\mathcal{A}$ be a finite subset of $X$. If the graph $G$ is a star graph, $\partial G = \mathcal{A}$ and $\mathcal{A}$ is the set of leaves of $G$, then the problem of finding all minimal parametric networks of type $G$ joining $\mathcal{A}$ is called the Fermat--Steiner problem on $\mathcal{A}$.
\end{definition}

Indeed, the Fermat--Steiner problem can be formulated more simply without using the concept of a network in a metric space.

\begin{definition}\label{dfn:Fer_St}
Let $X$ be a metric space and $\mathcal{A}$ be a finite subset of $X$. The Fermat--Steiner problem on $\mathcal{A}$ is the problem of finding all points of $X$ such that the sum of distances from each of them to the points of $\mathcal{A}$ is minimal.
\end{definition}

\begin{remark}\label{rk:Fer_St}
Thus, according to Proposition~\ref{prop:Fer_St}, \textbf{when studying the properties of an interior vertex {\boldmath $g(v)$} of a minimal parametric network {\boldmath $g$}, we will conveniently consider this vertex as one of the solutions to the Fermat--Steiner problem on the set of all vertices adjacent to {\boldmath $g(v)$} in the network {\boldmath $g$}.}
\end{remark}

\subsection{Fermat--Steiner Problem in Hyperspaces}

The purpose of this article is to study minimal parametric networks in hyperspaces $\mathcal{P}_0(X)$ of all nonempty arbitrary subsets of a metric space $X$ endowed with Hausdorff distance $d_H$. By virtue of the Proposition~\ref{prop:metr}, we may restrict ourselves to set of all closed subsets $\mathcal{P}_{\Cl}(X)$. In such case, a hyperspace $\mathcal{P}_{\Cl}(X)$ is an extended metric space. Henceforth, by $\mathcal{P}_{\Cl}^f(X)$ we will denote an arbitrary finiteness class of the hyperspace $\mathcal{P}_{\Cl}(X)$. 

More specifically, our further goal is to describe some properties of interior vertices of minimal parametric networks in $\mathcal{P}_{\Cl}^f(X)$. Therefore, by virtue of Remark~\ref{rk:Fer_St}, for convenience we will view such vertices as solutions to the Fermat--Steiner problem on the set of vertices adjacent to them in these networks. Namely, let $g$ be a minimal parametric network of type $G$ joining a finite set $\mathcal{A}\subset \mathcal{P}_{\Cl}^f(X)$. Let $v$ be an interior vertex of the graph $G$ and let $\{u_1,\ldots , u_n\}$ be the set of vertices adjacent to $v$ in $G$.

Let $\Sigma\subset \mathcal{P}_{\Cl}^f(X)$ be the set of all solutions of the Fermat--Steiner problem on $\bigl\{g(u_1),\ldots , g(u_n)\bigr\}$. First, by Proposition~\ref{prop:Fer_St}, we have $g(v)\in \Sigma$. Second, it is also obvious that for any $K\in \Sigma$ the network
\begin{equation}
g'(x) = \left\{
\begin{aligned}
    & g(x), && \text{if } x\neq v,\\
    & K,  && \text{if } x=v
\end{aligned}
\right.
\end{equation}
is a minimal parametric network of type $G$ joining $\mathcal{A}$. We introduce the notation
\begin{equation}
M_i:= g(u_i);
\end{equation}
\begin{equation}
\mathcal{M}:= \{M_1, \ldots, M_n\}.
\end{equation}
Based on the above, we will focus further on solving the Fermat--Steiner problem on $\mathcal{M}\subset \mathcal{P}_{\Cl}^f(X)$ in the formulation of Definition~\ref{dfn:Fer_St}. Thus, our task can be formulated as follows: we need to find all elements from $\mathcal{P}_{\Cl}^f(X)$ that minimize the function 
\begin{equation}
S_{\mathcal{M}}(Y) = \sum\limits_{i=1}^n d_H(Y, M_i).
\end{equation}

By analogy, the fixed finite set $\mathcal{M}$ will be called the \textit{boundary}, and each of its elements $M_i$ a \textit{boundary set}. Let $\Sigma(\mathcal{M})$ denote the set of all solutions of the Fermat--Steiner problem on the boundary $\mathcal{M}$. We introduce some more notation. Let $K\in \Sigma(\mathcal{M})$, then
\begin{equation}
d_i(K) := d_H(K, M_i);
\end{equation}
\begin{equation}
d(K) := \bigl(d_1(K),\ldots, d_n(K)\bigr);
\end{equation}
\begin{equation}
\Omega(\mathcal{M}) := \bigl\{d(K) : K\in \Sigma(\mathcal{M})\}.
\end{equation}

Note that for $\mathcal{M}\subset \mathcal{P}_{\Cl}^f(X)$, different elements $K$ and $K'$ from $\Sigma(\mathcal{M})$ can define the same vector from $\Omega(\mathcal{M})$, i.e., $d(K) = d(K').$ Moreover, for each element $K\in \Sigma(\mathcal{M})$, the vector $d(K)$ is uniquely determined. Therefore, the set of solutions $\Sigma(\mathcal{M})$ is divided into disjoint classes $\Sigma_d(\mathcal{M})$, each corresponding to a distinct vector $d = (d_1,\ldots, d_n)\in \Omega(\mathcal{M})$, where $d_i = d_H(K, M_i)$ for all $K\in \Sigma_d(\mathcal{M})$. Furthermore, each class $\Sigma_d(\mathcal{M})\subset \mathcal{P}_{\Cl}^f(X)$ is a partially ordered set with respect to inclusion. Note also, since $\mathcal{M}\subset \mathcal{P}_{\Cl}^f(X)$, then $d_i < \infty$ because of Proposition~\ref{prop:finite_len_on_bound_points}.

\section{Properties of Internal Vertices of Minimal Parametric Networks in Hyperspaces}\label{Main}
\subsection{Structure of $\Sigma_d(\mathcal{M})$}\label{chap:struct}

\begin{proposition}\label{prop:matreshka}
Let $X$ be a metric space, $\mathcal{M}=\{M_1, \ldots, M_n\}\subset \mathcal{P}_{\Cl}^f(X)$, $\Sigma(\mathcal{M})\neq \emptyset$ and $d\in \Omega(\mathcal{M})$. Let also $K_1, K_2\in \Sigma_d(\mathcal{M})$ such that $K_1\subset K_2$. For any $K\in \mathcal{P}_{\Cl}^f(X)$, if $K_1\subset K\subset K_2$, then $K\in \Sigma_d(\mathcal{M})$.
\end{proposition}

\begin{proof}

By Lemma~\ref{lem:inter} we obtain $d_H(K, M_i)\le d_H(K_1, M_i)$ for all $i$. But $K_1\in \Sigma_d(\mathcal{M})$, so $d_H(K, M_i) = d_H(K_1, M_i).$ Hence, $K\in \Sigma_d(\mathcal{M})$.\qed

\end{proof}

Let $X$ be a metric space, $\mathcal{M} = \{M_1, \ldots , M_n\}\subset \mathcal{P}_{\Cl}^f(X)$, and $\widetilde{d} = (\widetilde{d}_1, \ldots, \widetilde{d}_n)$ be a vector with non-negative components. Let 
\begin{equation}
K_{\widetilde{d}}(\mathcal{M}):= \bigcap\limits_{i=1}^n B_{\widetilde{d}_i}(M_i).
\end{equation}
Since each set $B_{\widetilde{d}_i}(M_i)$ is closed by Lemma~\ref{lem:Closed}, $K_{\widetilde{d}}(\mathcal{M})$ is also closed. However, note that $K_{\widetilde{d}}(\mathcal{M})$ may be empty. In what follows, since it is clear at every instance which boundary $\mathcal{M} = \{M_1,\dots,M_n\}$ is meant, we will write $K_{\widetilde{d}}$ instead of $K_{\widetilde{d}}(\mathcal{M})$.

\begin{proposition}\label{prop:great}
Let $X$ be a metric space, $\mathcal{M} = \{M_1,\ldots , M_n\}\subset \mathcal{P}_{\Cl}^f(X)$, $\Sigma(\mathcal{M})\neq \emptyset$, and $d\in \Omega(\mathcal{M})$. Then $K_d$ is the greatest element with respect to inclusion from $\Sigma_d(\mathcal{M})$.
\end{proposition}

\begin{proof}

Let $K\in \Sigma_d(\mathcal{M})$. Then, according to the definition of the Hausdorff distance in Definition~\ref{distance_H}, we have $K\subset B_{d_i}(M_i)$ for all $i$. Therefore, $K\subset \bigcap\limits_{i=1}^n B_{d_i}(M_i) = K_d$. Let us now show that $K_d\in \Sigma_d(\mathcal{M})$. 

By definition, $K_d\subset B_{d_i}(M_i)$. But since $K\in \Sigma_d(\mathcal{M})$, we have $M_i\subset B_{d_i}(K)\subset B_{d_i}(K_d)$. Therefore, $d_H(K_d, M_i) \le d_i$. But since $K\in \Sigma_d(\mathcal{M})$, we have 
\begin{equation}
d_H(K_d, M_i) = d_i.
\end{equation}
Thus, $K_d\in \Sigma_d(\mathcal{M})$ and, for any $K\in \Sigma_d(\mathcal{M})$, we have $K\subset K_d$.\qed

\end{proof}

\begin{remark}
It is obvious that in any partially ordered set, if a greatest element exists, then it is unique. Therefore, by Proposition~\ref{prop:great}, if $d\in \Omega(\mathcal{M})$, then $K_d$ is the unique greatest element with respect to inclusion in the class $\Sigma_d(\mathcal{M})$.
\end{remark}

The proof of the following statement repeats verbatim the proof of Assertion~3.1 from~\cite{ITT}.

\begin{proposition}\label{prop:min}
Let $X$ be a metric space, $\mathcal{M} = \{M_1,\ldots , M_n\}\subset \mathcal{P}_{\Cl}^f(X)$, $\Sigma(\mathcal{M})\neq \emptyset$, and $d\in \Omega(\mathcal{M})$. Then each compact set $K\in \Sigma_d(\mathcal{M})$ contains at least one minimal element with respect to inclusion in $\Sigma_d(\mathcal{M})$.
\end{proposition}

\begin{proposition}\label{prop:min_cl}
Let $X$ be a proper metric space, $\mathcal{M} = \{M_1,\ldots , M_n\}\subset \mathcal{P}_{\Cl}^f(X)$, $\Sigma(\mathcal{M})\neq \emptyset$ and $d\in \Omega(\mathcal{M})$. Then each $K\in \Sigma_d(\mathcal{M})$ contains at least one minimal element with respect to inclusion in $\Sigma_d(\mathcal{M})$.
\end{proposition}

\begin{proof}

If $K$ is finite, then it is obvious that $K$ contains a minimal element with respect to inclusion in $\Sigma_d(\mathcal{M})$. Now suppose $K$ is infinite.

Consider the set $\mathcal{C} = \bigl\{Y\in \Sigma_d(\mathcal{M}) \colon Y\subset K\bigr\}$ naturally ordered with respect to inclusion. Since $K\in \mathcal{C}$, we have $\mathcal{C}\neq \emptyset$. Let us show that each chain $\{K_{\alpha}\}_{\alpha\in \Lambda}\subset \mathcal{C}$ possesses a lower bound, and apply Zorn's Lemma. A natural candidate for the lower bound is $K_* := \bigcap\limits_{\alpha\in \Lambda} K_{\alpha}$. We will further assume that the chain $\{K_{\alpha}\}_{\alpha\in \Lambda}$ is ordered by decreasing inclusion relative to the growth of the indices. Let us show that $K_* := \bigcap\limits_{\alpha\in \Lambda} K_{\alpha}\neq \emptyset$ and $M_j\subset B_{d_j}(K_*)$.

So, $K_{\alpha}\subset K\in \Sigma_d(\mathcal{M})$. Hence, we have $d_H(K_{\alpha}, M_j) = d_j$ for all $\alpha$. Consequently,
\begin{equation}\label{eq:inclusion}
M_j\subset B_{d_j}(K_{\alpha})
\end{equation}
for all $\alpha$ and $j$. 

Choose a point $x$ in $M_j$. Since $X$ is a proper metric space and $K_{\alpha}$ is closed then there exists $k_{\alpha}\in K_{\alpha}$ such that $|x\, k_{\alpha}| = |x\, K_{\alpha}|$. So, by~(\ref{eq:inclusion}), we obtain $|x\, k_{\alpha}| \le d_j$, and therefore $\{k_{\alpha}\}\subset B_{d_j}(x)$.

Since $X$ is a proper metric space, $B_{d_j}(x)$ is compact. So, in $B_{d_j}(x)$, one can extract a convergent subnet from every net. Let $\{k_{{\alpha}_{\beta}}\}$ be a subnet which converges to some $k$. Since a chain is a linear ordered set, then, for any $\alpha_0$, we have $\{k_{{\alpha}_{\beta}} \colon \alpha_{\beta}\succeq \alpha_0\}\subset K_{\alpha_0}$. Consequently, by the closedness of $K_{\alpha_0}$ we obtain $k\in K_{\alpha_0}$ for all $\alpha_0$. Therefore $k\in \bigcap\limits_{\alpha\in \Lambda} K_{\alpha} = K_*$. Moreover, from the continuity of the distance function in a metric space it follows that $|x\, k|\le d_j$.

Thus, since the choice of the point $x\in M_j$ was arbitrary, we obtain $M_j\subset B_{d_j}(K_*)$. At the same time, since $d_H(K_{\alpha}, M_j) = d_j$ and $K_*\subset K_{\alpha}$, we have $K_*\subset B_{d_j}(M_j)$. Therefore $d_H(K_*, M_j)\le d_j$. Note that $K_*$ is closed because all $K_{\alpha}$ are closed. Consequently, $K_*$ belongs to the same finiteness class $\mathcal{P}_{\Cl}^f(X)$ as all $M_j$. But since $\min\limits_{Y\in \mathcal{P}_{\Cl}^f(X)} S_{\mathcal{M}}(Y) = \sum\limits_{j=1}^n d_j$, we have $d_H(K_*, M_j) = d_j$. Hence, $K_*\in \Sigma_d(\mathcal{M})$ and $K_*$ is a lower bound of the chain $\{K_{\alpha}\}_{\alpha\in \Lambda}\subset \mathcal{C}$. 

Thus, by Zorn's Lemma (in the formulation for minimal elements), $\mathcal{C}$ has a minimal element. This element is the required minimal (with respect to inclusion) solution contained in $K$.\qed

\end{proof}

\subsection{Implementation of One-Sided Hausdorff Distances on Interior Vertices}\label{chap:one_dist}

Let us write down the following definition.

\begin{definition}\label{dfn:far}
Let $X$ be a metric space, $\mathcal{M} = \{M_1, \ldots, M_n\}\subset \mathcal{P}_{\Cl}^f(X)$ and $d = (d_1, \ldots, d_n)\in \Omega(\mathcal{M})$. A point $x\in M_i$ is called $d$-far for $M_i$ if $U_{d_i}(x)\cap K_d = \emptyset$.
\end{definition}

In the paper~\cite[Theorem~3]{Gals_1} it was proved that in the case of a finite-dimensional normed $X$ and all nonempty convex compacts $M_i\in \mathcal{M}$, there exists an $M_j\in \mathcal{M}$ that has a $d$-far point. This fact essentially relied on the compactness of the sets $M_i$.

Note that from $U_{d_i}(x)\cap K_d = \emptyset$ and $d_H(M_i, K_d) = d_i$ it follows that $\sup\limits_{x\in M_i} |x\, K_d| = d_i$. Moreover, by the compactness of $M_i$ and the continuity of the distance function to a nonempty subset, from $\sup\limits_{x\in M_i} |x\, K_d| = d_i$ it follows that there exists an $x\in M_i$ such that $|x\, K_d| = d_i$ and hence $U_{d_i}(x)\cap K_d = \emptyset$.

Thus, for a finite-dimensional normed $X$, nonempty set of convex compacts \linebreak $\mathcal{M} = \{M_1,\ldots , M_n\}$ and $d\in \Omega(\mathcal{M})$, there always exists a compact $M_i$ such that $d_i = \sup\limits_{x\in M_i} |x\, K_d|$. That is, one of the Hausdorff distances in the sum $S_{\mathcal{M}}(K_d) = \sum\limits_{i=1}^n d_H(M_i, K_d)$ will be one-sided precisely from the side of the boundary set.

We will prove under some additional conditions that if we omit the compactness of the sets $M_i$ and the finite-dimensionality of the space $X$, then for $d = (d_1, \ldots, d_n)\in \Omega(\mathcal{M})$ there still exists an $i$ such that $d_i = \sup\limits_{x\in M_i} |x\, K_d|$.

Similarly, we will denote an arbitrary finiteness class of the hyperspace \linebreak $\mathcal{P}_{\Cl, \Conv}(X)$ by $\mathcal{P}_{\Cl, \Conv}^f(X)$.

\begin{theorem}\label{thm:one_side}
Let $X$ be a real normed space, 
\begin{equation}
\mathcal{M} = \{M_1,\ldots , M_n\}\subset \mathcal{P}_{\Cl, \Conv}^f(X),
\end{equation}
$\Sigma(\mathcal{M})\neq \emptyset$ and $d\in \Omega(\mathcal{M})$. Suppose also that there exists an index $s$ such that all nonempty sets of the form $B_r(M_s)\cap K_d$ belong to the same finiteness class $\mathcal{P}_{\Cl, \Conv}^f(X)$ to which all $M_i$ belong. Then there exist $i$ such that $d_i = \sup\limits_{x\in M_i} |x\, K_d|$.
\end{theorem}

\begin{proof}

Suppose the contrary, i.e., that for all $M_i$, we have 
\begin{equation}\label{eq:positive}
d_i > \sup\limits_{x\in M_i} |x\, K_d|\ge 0. 
\end{equation}
This means that there exists $0 < \varepsilon_i < d_i - \sup\limits_{x\in M_i} |x\, K_d|$ such that for all points $x\in M_i$ we have $U_{d_i - \varepsilon_i}(x)\cap K_d\neq \emptyset$. But then for any $i$,
\begin{equation}
M_i\subset U_{d_i - \varepsilon_i}(K_d)\subset B_{d_i - \varepsilon_i}(K_d)\subset B_{\varepsilon_i}\bigl(B_{d_i - \varepsilon_i}(K_d)\bigr).
\end{equation}
By Lemma~\ref{bound} we have $\Bigl|B_{d_i - \varepsilon_i}(K_d) \,\, \partial B_{\varepsilon_i}\bigl(B_{d_i - \varepsilon_i}(K_d)\bigr)\Bigr| = \varepsilon_i > 0$. Moreover, by Lemma~\ref{sum_1} we have $B_{\varepsilon_i}\bigl(B_{d_i - \varepsilon_i}(K_d)\bigr) = B_{d_i}(K_d)$. Thus, $\bigl|M_i\, \partial B_{d_i}(K_d)\bigr| > 0$. Hence, because of~(\ref{eq:positive}), the following quantity is positive:
\begin{equation}
\gamma_i = \min\Bigl\{d_i, \bigl|M_i\, \partial B_{d_i}(K_d)\bigr|\Bigr\}.
\end{equation}
Consequently, since the number of boundary sets is finite, we obtain
\begin{equation}
\gamma = \min\limits_{i\in [1, n]} \gamma_i > 0.
\end{equation}
By condition, all $M_i$ are convex. Then by Lemma~\ref{lem:convexity}, the sets $B_{d_i}(M_i)$ are convex. Hence, $K_d$ is also convex. So, for any $i$, by Proposition~\ref{prop:free_space} we have
\begin{equation}\label{eq:reduction}
M_i \subset B_{d_i-\gamma}(K_d).
\end{equation}
In particular, expression~(\ref{eq:reduction}) holds for the index $s$ from the condition of the theorem. Recall that $d_H(M_s, K_d) = d_s < \infty$. Since $d_s > \sup\limits_{x\in M_s} |x\, K_d|$, then we have $d_s\in \bigl(|M_s\, K_d|, \infty\bigr)$, due to Lemma~\ref{lem:equal}. By the condition, for all $r\in \bigl(|M_s\, K_d|, \infty\bigr)$, sets $B_r(M_s)\cap K_d$ belong to the same finiteness class $\mathcal{P}_{\mathrm{\Cl, \Conv}}^f(X)$ as all $M_i$. Therefore, by Theorem~\ref{cont_criterion}, the set $B_{d_s}(M_s)\cap K_d$ varies continuously in the Hausdorff sense under small perturbations of $d_s$.

Thus, for $0<\varepsilon<\gamma$ there exists $\delta>0$ such that
\begin{equation}
K_d = B_{d_s}(M_s)\cap K_d \subset B_{\varepsilon}\bigl(B_{d_s-\delta}(M_s)\cap K_d\bigr).
\end{equation}
Denote $B_{d_s-\delta}(M_s)\cap K_d$ by $K$. Thus, $K_d\subset B_{\varepsilon}(K)$. Consequently, for any $i$, by Lemma~\ref{sum_1},
\begin{equation}
M_i \subset B_{d_i-\gamma}(K_d) \subset B_{d_i-\gamma}\bigl(B_{\varepsilon}(K)\bigr) = B_{d_i-\gamma+\varepsilon}(K).
\end{equation}
Moreover, $K\subset K_d\subset B_{d_i}(M_i)$ for all $i$ and $K\subset B_{d_s-\delta}(M_s)$. Hence $d_H(K, M_i)\le d_i$ and $d_H(K, M_s)\le \max(d_s-\delta, d_s-\gamma+\varepsilon) < d_s$. Consequently, $S_{\mathcal{M}}(K) < S_{\mathcal{M}}(K_d)$, which is a contradiction. The theorem is proved.\qed

\end{proof}

\begin{proposition}\label{prop:one_side}
Let $X$ be a metric space, $\mathcal{M} = \{M_1,\ldots , M_n\}\subset \mathcal{P}_{\Cl}^f(X)$, $\Sigma(\mathcal{M})\neq \emptyset$ and $d\in \Omega(\mathcal{M})$. Suppose also that $d_i = \sup\limits_{x\in M_i} |x\, K_d|$. Then, for any $K\in \Sigma_d(\mathcal{M})$, it holds $d_i = \sup\limits_{x\in M_i} |x\, K|$.
\end{proposition}

\begin{proof}

Since $K\in \Sigma_d(\mathcal{M})$, then $d_i = d_H(M_i, K) = \max\{\sup\limits_{x\in M_i} |x\, K|, \sup\limits_{x\in K} |x\, M_i|\}\ge \sup\limits_{x\in M_i} |x\, K|$. But because of $K\subset K_d$, we have $d_i = \sup\limits_{x\in M_i} |x\, K_d|\le \sup\limits_{x\in M_i} |x\, K|$. Thus, $d_i = \sup\limits_{x\in M_i} |x\, K|$.\qed

\end{proof}

The next corollary follows from Theorem~\ref{thm:one_side} and Proposition~\ref{prop:one_side}.

\begin{corollary}\label{cor:one_side}
Let $X$ be a real normed space, 
\begin{equation}
\mathcal{M} = \{M_1,\ldots , M_n\}\subset \mathcal{P}_{\Cl, \Conv}^f(X), 
\end{equation}
$\Sigma(\mathcal{M})\neq \emptyset$ and $d\in \Omega(\mathcal{M})$. Suppose also that there exists an index $s$ such that all nonempty sets of the form $B_r(M_s)\cap K_d$ belong to the same finiteness class $\mathcal{P}_{\Cl, \Conv}^f(X)$ to which all $M_i$ belong. Then, for any $K\in \Sigma_d(\mathcal{M})$, there exist $i$ such that $d_i = \sup\limits_{x\in M_i} |x\, K|$.
\end{corollary}

Also according to Theorem~\ref{thm:2_dim} and Corollary~\ref{cor:one_side}, we also have the following corollary.

\begin{corollary}\label{cor:one_side_2_dim}
Let $X$ be a real normed space, dimension of $X$ does not exceed $2$, 
\begin{equation}
\mathcal{M} = \{M_1,\ldots , M_n\}\subset \mathcal{P}_{\Cl, \Conv}^f(X), 
\end{equation}
$\Sigma(\mathcal{M})\neq \emptyset$ and $d\in \Omega(\mathcal{M})$. Then, for any $K\in \Sigma_d(\mathcal{M})$, there exist $i$ such that $d_i = \sup\limits_{x\in M_i} |x\, K|$.
\end{corollary}

Also note that, for any $A, B\in \mathcal{P}_{\Cl, \B, \Conv}(X)$, we have $d_H(A, B) < \infty$. Hence, from Corollary~\ref{cor:one_side}, we obtain the following corollary.

\begin{corollary}\label{cor:one_side_bound}
Let $X$ be a real normed space, 
\begin{equation}
\mathcal{M} = \{M_1,\ldots , M_n\}\subset \mathcal{P}_{\Cl, \B, \Conv}(X), 
\end{equation}
$\Sigma(\mathcal{M})\neq \emptyset$ and $d\in \Omega(\mathcal{M})$. Then, for any $K\in \Sigma_d(\mathcal{M})$, there exist $i$ such that $d_i = \sup\limits_{x\in M_i} |x\, K|$.
\end{corollary}

Since $\mathcal{P}_{\Comp, \Conv}(X)\subset \mathcal{P}_{\Cl, \B, \Conv}$ and a continuous function on a compact set always attains its supremum, from Corollary~\ref{cor:one_side_bound} we also obtain the following corollary.

\begin{corollary}
Let $X$ be a real normed space, 
\begin{equation}
\mathcal{M} = \{M_1,\ldots , M_n\}\subset \mathcal{P}_{\Comp, \Conv}(X), 
\end{equation}
$\Sigma(\mathcal{M})\neq \emptyset$ and $d\in \Omega(\mathcal{M})$. Then, for any $K\in \Sigma_d(\mathcal{M})$, there exist $i$ and $x\in M_i$ such that $U_{d_i}(x)\cap K = \emptyset$. In particular, there exists $M_i\in \mathcal{M}$ which has $d$-far point.
\end{corollary}

A natural question arises: for $K\in \Sigma_d(\mathcal{M})$, can all distances in the functional $S_{\mathcal{M}}(K)$ have the view $d_i = \sup\limits_{x\in M_i} |x\, K| > \sup\limits_{x\in K} |x\, M_i|$. It turns out that, for example, this is not true in a connected metric space for the set $K_d\in \Sigma_d(\mathcal{M})$.

\begin{theorem}\label{thm:one_side_two}
Let $X$ be a connected metric space, 
\begin{equation}
\mathcal{M} = \{M_1,\ldots , M_n\}\subset \mathcal{P}_{\Cl}^f(X), 
\end{equation}
$\Sigma(\mathcal{M})\neq \emptyset$ and $d\in \Omega(\mathcal{M})$. Then there exists $i$ such that $d_i = \sup\limits_{x\in K_d} |x\, M_i|$.
\end{theorem}

\begin{proof}

If $K_d = X$, then, since $M_i\subset X$, we obtain 
\begin{equation}
d_i = d_H(M_i, X) = \max\{\sup_{y\in M_i} |y\, X|, \sup_{x\in X} |x\, M_i|\} = \sup_{x\in X} |x\, M_i|.
\end{equation}

Now suppose $K_d \neq X$. Assume that for all $i$ the strict inequality $\sup_{x\in K_d} |x\, M_i| < d_i$ holds. This means that $K_d \subset U_{d_i}(M_i)$ for all $i$. Consequently, $K_d \subset \bigcap_{i=1}^n U_{d_i}(M_i)$. Moreover, $\bigcap_{i=1}^n U_{d_i}(M_i) \subset \bigcap_{i=1}^n B_{d_i}(M_i) = K_d$. Hence,
\begin{equation}
K_d = \bigcap_{i=1}^n U_{d_i}(M_i).
\end{equation}
But then, by Lemmas~\ref{lem:Closed} and~\ref{lem:Opened}, the set $K_d$ is simultaneously closed and open. In view of the connectedness of the space $X$, any simultaneously closed and open subset of $X$ coincides with $\emptyset$ or with $X$. However, $K_d$ is nonempty and by assumption $K_d \neq X$. We obtain a contradiction. Therefore, there exists an $i$ such that $d_i = \sup\limits_{x\in K_d} |x\, M_i|$.\qed

\end{proof}

%
%

\end{document}